\documentclass[9pt,a4paper,hidelinks]{extarticle}

\newif\ifLONGVER
\LONGVERtrue

\usepackage{CJKutf8}
\usepackage[utf8]{inputenc} 
\usepackage{amsthm}
\usepackage{amsfonts}
\usepackage{scalerel}
\usepackage{amssymb}
\usepackage{ mathrsfs }
\usepackage{mathtools}
\usepackage{xcolor}
\usepackage{colonequals}
\usepackage{xfrac}
\usepackage{caption}
\usepackage{environ}
\usepackage{tikz}
\usepackage{adjustbox}
\usepackage{booktabs,tabularx}
\usepackage[normalem]{ulem}
\usepackage[shortlabels]{enumitem}
\usepackage{setspace}
\usepackage[hyphens]{url}
\usepackage{hyperref}
\usepackage[noabbrev]{cleveref}
\usepackage[citestyle=alphabetic,
style=alphabetic,
giveninits=true,
natbib=true,
backend=biber,
]
{biblatex}

\usepackage[nottoc]{tocbibind}

\DeclareSourcemap{
  \maps[datatype=bibtex]{
    \map[overwrite=true]{
      \step[fieldset=url, null]
	  \step[fieldset=doi, null]
	  \step[fieldset=isbn, null]
	  \step[fieldset=issn, null]
    }
  }
}
\newtheorem{theorem}{Theorem}[section]
\newtheorem{lemma}[theorem]{Lemma}
\newtheorem{proposition}[theorem]{Proposition}
\newtheorem{corollary}[theorem]{Corollary}

\theoremstyle{definition}
\newtheorem{definition}[theorem]{Definition}

\theoremstyle{remark}
\newtheorem{remark}[theorem]{Remark}
\newtheorem{notation}[theorem]{Notation}

\numberwithin{equation}{section}

\setlist[enumerate,itemize]{leftmargin=0.75cm,itemsep=1.5pt,topsep=1.5pt}
\newcommand{\lang}[2]{\mscr{L}_{{\scriptscriptstyle #1}}^{{\scriptscriptstyle #2}}}

\newcommand{\bb}{\mathbb}

\newcommand{\mcal}{\mathcal}
\newcommand{\mscr}{\mathscr}

\newcommand{\mfrak}{\mathfrak}

\newcommand{\ul}{\underline}

\newcommand{\pow}{\mcal{P}}

\newcommand{\wo}{\setminus}

\newcommand{\upperRomannumeral}[1]{\uppercase\expandafter{\romannumeral#1}}

\makeatletter

\newcommand{\Rmn}[1]{\expandafter\@slowromancap\romannumeral #1@}
\makeatother
\makeatletter
\newcommand{\oset}[3][0ex]{%
  \mathrel{\mathop{#3}\limits^{
    \vbox to#1{\kern-2\ex@
    \hbox{$\scriptstyle#2$}\vss}}}}
\makeatother
\newcommand{\exle}{\oset[0ex]{\scriptscriptstyle \exists}{<}}

\DeclareMathOperator{\suf}{s}
\DeclareMathOperator{\prf}{p}
\newcommand{\ps}{\suf^{-1}}
\newcommand{\pp}{\prf^{-1}}
\newcommand{\ips}{\mfrak{s}^{-1}}

\newcommand{\lcro}{\mathbb{I}}

\newcommand{\str}[1]{\mscr{#1}} 

\newcommand{\langMO}{\lang{\MO}{}}
\newcommand{\langMFIN}{\lang{\MSO(\Fin)}{}}
\newcommand{\langWI}{\lang{\WSO(\lcro)}{}}
\newcommand{\langLI}{\lang{\LCI(\lcro)}{}}

\newcommand{\bdr}[1]{#1^{\partial}}
\newcommand{\cz}{\mathbf{0}}

\newcommand{\newsection}[1]{\section{#1}}
\newcommand{\newsubsection}[1]{\subsection{#1}}

\DeclarePairedDelimiter\dVert{\lVert}{\rVert}

\DeclarePairedDelimiter\funky{\lfloor}{\rceil}
\DeclareMathOperator{\Th}{Th}

\DeclareMathOperator{\At}{At}

\DeclareMathOperator{\Fin}{Fin}
\DeclareMathOperator{\fin}{fin}

\DeclareMathOperator{\MO}{MO}

\DeclareMathOperator{\WSO}{W}
\DeclareMathOperator{\MSO}{M}
\DeclareMathOperator{\LCI}{L}

\DeclareMathOperator{\bdd}{Bd}

\DeclareMathOperator{\fci}{fci}

\DeclareMathOperator{\ind}{ind}

\begin{document}
\title{Model-completeness for a dense linear order in weak monadic second order logic}
\author{Deacon Linkhorn}
\date{\today}

\maketitle

\vspace{-0.5cm}

\tableofcontents

\vspace{-0.5cm}

\clearpage
\newsection{Introduction}

In this note we present a streamlined and (hopefully) accessible proof of the main result from the authors' PhD thesis \cite{DLPhD}.
The result, \cref{Theorem-WIModelComplete}, establishes model-completeness of the weak monadic second order version of a dense linear order with left-endpoint but no right-endpoint in the signature $\langWI$ (\cref{Definition-LWI}).
In essence the proof given here and the proof given in \cite{DLPhD} are the same, both follow the same pattern of using a model-completeness result for the pseudofinite monadic second order theory of linear order together with Feferman-Vaught machinery of generalised powers.
However, the presentation here is greatly simplified, mostly through the removal of unnecessary (with respect to the model-completeness result) computational aspects which were considered in much of the supporting machinery in \cite{DLPhD}.
The role of the Feferman-Vaught machinery is much more transparent in the presentation given here.

The proof of the model-completeness result for the weak monadic second order version of a linear order occupies the four sections of the note following this introduction. 
In the final section it is shown that working in a signature $\langLI$ (\cref{Definition-LI}), model-completeness for the lattice of finite unions of closed intervals for this dense linear order can be obtained using the model-completeness result for the weak monadic second order version using interpretations.

There is quite substantial overlap with the note/preprint \cite{DeaconLIpreprint}. 
We have reused some of the preliminary material given in \cref{Section-definitions}, and \cref{Section-transfer} is simply a reproduction of the main results from \cite{DeaconLIpreprint}.

Logical terminology and notation, where not explicitly defined, tries to follow \cite{HodgesBible}.

\newsection{Setup}\label{Section-definitions}

In this section we will setup some definitions which we need as well as recalling some results which we make use of.

We will first introduce the notion of the (weak) monadic second order version of a linear order. 
Then we will look at two theories of (weak) monadic second order versions of linear orders, giving signatures for both which are suitable to achieve model-completeness results. 

\newsubsection{Basic concepts}

\begin{definition}
Let $S$ be a set.
We write $\pow(S)$ for the powerset of $S$.
We write $\pow_{\fin}(S)$ for the collection of finite subsets of $S$.
\end{definition}

\begin{definition}
Let $\alpha$ be any linear order.
A \textbf{closed interval} of $\alpha$ is a subset $S \subseteq \alpha$ of one of the following forms,
\begin{enumerate}
\item $[i,j] \coloneqq \{k \in \alpha: i \leq k \leq j\}$ for $i,j \in \alpha$,
\item $[i,+\infty) \coloneqq \{k \in \alpha: i \leq k\}$ for $i \in \alpha$,
\item $(-\infty,j] \coloneqq \{k \in \alpha: k \leq j\}$ for $j \in \alpha$.
\end{enumerate}
We will write $\pow_{\fci}(\alpha)$ for the set of finite unions of closed intervals of $\alpha$.
\end{definition}

Note that $\pow_{\fin}(\alpha) \subseteq \pow_{\fci}(\alpha) \subseteq \pow(\alpha)$.

\begin{definition}\label{Def-discreteassociatedfunc}
Now let $\alpha$ be a \emph{discrete} linear order. We associate the following four functions with $\alpha$,
\begin{enumerate}
    \item $\suf_{\alpha}: \alpha \wo \max(\alpha) \rightarrow \alpha$ for the successor function on $\alpha$,
    \item $\prf_{\alpha}: \alpha \wo \min(\alpha) \rightarrow \alpha$ for the predecessor function on $\alpha$,
    \item $\ps_{\alpha}: \pow(\alpha) \rightarrow \pow(\alpha)$,  $A  \mapsto \{i \in \alpha: \suf_{\alpha}(i) \in A\}$,
    \item $\pp_{\alpha}: \pow(\alpha) \rightarrow \pow(\alpha)$,  $A  \mapsto \{i \in \alpha: \prf_{\alpha}(i) \in A\}$.
\end{enumerate}

With $\alpha$ again a linear order (not necessarily discrete), we call $A \in \pow(\alpha)$ \textbf{discrete} if the restriction of $\alpha$ to $A$ is a discrete linear order.
For such discrete $A \subseteq \alpha$ we write $\suf_A$ (respectively $\prf_A$) for the successor (respectively predecessor) function on the restriction of $\alpha$ to a linear ordering on $A$.
Note that every finite subset of $\alpha$ is discrete. 
\end{definition}

\begin{definition}
We write $\langMO$ for the signature $\{\subseteq,\exle\}$ comprising two binary relation symbols. 
To each linear order $\alpha$ we associate two $\langMO$-structures.

The \textbf{monadic second order version} of $\alpha$, denoted by $\MSO(\alpha)$, has universe $\pow(\alpha)$. 
The \textbf{weak monadic second order version} of $\alpha$, denoted by $\WSO(\alpha)$, has universe $\pow_{\fin}(\alpha)$.
In both structures $\subseteq$ is interpreted as set-theoretic inclusion, while $\exle$ is interpreted as follows,
\[
A \exle B \text{ if and only if } i < j \text{ for some }i \in A,j \in B,
\]
where $<$ is the ordering of $\alpha$.
\end{definition}

\begin{remark}
For each linear order $\alpha$, the collection of singleton subsets of $\alpha$ is definable in both $\MSO(\alpha)$ and $\WSO(\alpha)$. 
The same $\langMO$-formula works in both structures, and is given by taking the atoms of the ordering $\subseteq$.
As such we will write $\At$ to denote both the collection of singletons and the formula which defines this collection.
\end{remark}

\begin{definition}\label{Definition-DefEquiv}
Let $\lang{1}{}$ and $\lang{2}{}$ be first order signatures. 
We will say that an $\lang{1}{}$-structure $\str{M}$ and an $\lang{2}{}$-structure $\str{N}$ are \textbf{definitionally equivalent} over a bijection $f:\str{M} \leftrightarrow \str{N}$ if (after identifying the underlying sets via the bijection $f$) the definable sets in $\str{M}$ and $\str{N}$ coincide.
A necessary and sufficient condition for $\str{M}$ and $\str{N}$ being definitionally equivalent is that for each \emph{atomic} $\lang{1}{}$-formula $\phi(\bar{x})$ there is a $\lang{2}{}$-formula $\psi(\bar{y})$ defining the same set as $\phi$ (after identification via $f$), and vice versa.

If $T_1$ is an $\lang{1}{}$-theory and $T_2$ is an $\lang{2}{}$-theory then we will say that $T_1$ and $T_2$ are definitionally equivalent if the following hold,
\begin{enumerate}
    \item for each $\str{M} \models T_1$ there is a unique $\str{N} \models T_2$ (w.l.o.g with the same universe, so that $f$ is trivial) with $\str{M}$ and $\str{N}$ definitionally equivalent, and vice versa,
    \item this occurs uniformly in the sense that for each atomic $\lang{1}{}$-formula $\phi(\bar{x})$ there is a $\lang{2}{}$-formula $\psi(\bar{x})$ such that for every $\str{M} \models T_1$ the formula $\psi$ defines the same set in the corresponding $\lang{2}{}$-structure $\str{N}$ as $\phi$ does in $\str{M}$, and vice versa.
\end{enumerate}

\end{definition}

\newsubsection{Restriction and relativisation}

We can often interpret (weak) monadic versions of linear orders inside one another. 
Here we will outline a common pattern that these interpretations take. 

Let $\alpha$ be a linear order. 
Let $\phi(X)$ be an $\langMO$-formula which defines a collection of singletons in either $\WSO(\alpha)$ or $\MSO(\alpha)$.
We can think of this collection of singletons as a subset of $\alpha$, and hence as a sub-order $\beta \subseteq \alpha$.
It is then not hard to see that $\pow_{\fin}(\beta)$, respectively $\pow(\beta)$, is definable in $\WSO(\alpha)$, respectively $\MSO(\alpha)$. 

$\WSO(\beta)$ is the substructure of $\WSO(\alpha)$ induced on $\pow_{\fin}(\beta)$, and likewise $\MSO(\beta)$ is the substructure of $\MSO(\alpha)$ induced on $\pow(\beta)$. 
Therefore $\WSO(\beta)$ is interpretable in $\WSO(\alpha)$, and $\MSO(\beta)$ is interpretable in $\MSO(\alpha)$.
These are particularly simple interpretations, restriction to a definable set in a relational signature.

Let us look at two pertinent special cases.

\begin{definition}\label{Definition-ResElement}
Let $\alpha$ be a linear order and suppose $A \in \MSO(\alpha)$.
Then taking the formula $\phi(X)$ to be $\At(X) \wedge X \subseteq A$ we can define (over the parameter $A$) the collection of singleton subsets of $\alpha$ contained in $A$. 
Viewing $A \subseteq \alpha$ as a suborder, we can therefore interpret $\MSO(A)$ in $\MSO(\alpha)$.
We will write $\MSO(\alpha) \upharpoonright A$ for the copy of $\MSO(A)$ living inside $\MSO(\alpha)$ given by this interpretation.
By standard arguments using the relativisation of quantifiers, for each $\langMO$-formula $\psi(\bar{Y})$ there is an $\langMO$-formula $\psi (\bar{Y}) \downharpoonright A$ such that,
\[
\MSO(\alpha) \upharpoonright A \models \psi(\bar{B}) \Longleftrightarrow \MSO(\alpha) \models \psi(\bar{B}) \downharpoonright A,
\]
for each $A \in \MSO(\alpha)$ and $\bar{B} \in \MSO(A)$.
\end{definition}

\begin{remark}
The above construction can be carried out under much weaker conditions, in particular for a much wider class of $\langMO$-structures.
For example, if we take $\str{M}$ an $\langMO$-structure such that $\str{M} \equiv \MSO(\alpha)$ for some $\alpha$, then $\str{M}$ must be an atomic distributive lattice. 
For any definable set of atoms $\rho(X)$ in $\str{M}$ it makes sense to consider the restriction of $\str{M}$ to that definable set, by considering the substructure induced on the set defined by,
\[
\mu(Y): \forall X ((\At(X) \wedge X \subseteq Y) \rightarrow \rho(X)).
\]
\end{remark}

\begin{definition}\label{Definition-ResInterval}
Let $\alpha$ be a linear order and suppose $i,j \in \alpha$.
Then taking the formula $\phi(X)$ to be $\At(X) \wedge (\{i\} \exle X \vee \{i\} = X) \wedge X \exle \{j\}$ we can define (over the parameters $\{i\}$ and $\{j\}$) the collection of singleton subsets of $\alpha$ contained in the interval $[i,j)$ of $\alpha$. 
Viewing $[i,j)$ as a suborder of $\alpha$, we can therefore interpret $\WSO([i,j))$ in $\WSO(\alpha)$.
Note that if we worked in $\MSO(\alpha)$ this would be a special case of the restricting to an element $A \in \MSO(\alpha)$, but in $\WSO(\alpha)$ this is not the case as in general an interval will not be a finite set.
We will write $\WSO(\alpha) \upharpoonright [i,j)$ for the copy of $\WSO([i,j))$ living inside $\WSO(\alpha)$ given by this interpretation.
Again by standard arguments using the relativisation of quantifiers, for each $\langMO$-formula $\psi(\bar{Y})$ there is an $\langMO$-formula $\psi (\bar{Y}) \downharpoonright A$ such that,
\[
\WSO(\alpha) \upharpoonright [i,j) \models \psi(\bar{B}) \Longleftrightarrow \WSO(\alpha) \models \psi(\bar{B}) \downharpoonright [i,j),
\]
for each $i,j \in \alpha$ and $\bar{B} \in \WSO([i,j))$.

As in \cref{Definition-ResElement}, this construction can be generalised to a much broader class of structures that the (weak) monadic second order versions of linear orders.
\end{definition}

We will make use of both of these particular cases, including their generalisation to non-standard structures (those not of the form $\WSO(\alpha)$ or $\MSO(\alpha)$ for some $\alpha$), in \cref{Section_WIModelComplete,Section_FefVau}.

\newsubsection{The pseudofinite monadic second order theory of linear order}

\begin{definition}\label{Definition-CanonicalFiniteLO}
For each $n \in \bb{N}$, we define $\ul{n}$ to be the linear order with underlying set $\{0,\ldots,n-1\}$ and with the usual ordering $0 < 1 < \ldots < n-1$. Note that we include $0 \in \bb{N}$, with $\ul{0}$ being the unique (trivial) linear ordering on the empty set.

The \textbf{pseudofinite monadic second order theory of linear order} is the shared $\langMO$-theory of the monadic second order versions of finite linear orders,
\[
\bigcap_{n \in \bb{N}}\Th(\MSO(\ul{n})).
\]

We write $T_{\MSO(\Fin)}$ for this theory.\footnote{This is the notation introduced in \cite{DLPhD}, where an explicit axiomatisation is given, as well as an analysis of the non-standard completions and proof of the model-completeness result which we give here without proof.}
\end{definition}

\begin{definition}
We write $\langMFIN$ for the signature $\{\cup,\cap,\bot,\cz,\cz^*,\ps\}$.
For each $n \in \bb{N}$, it is straightforward to see that the $\langMFIN$-structure with,
\begin{enumerate}
    \item universe $\pow(\ul{n})$,
    \item $\bot$ interpreted as the empty set,
    \item $\cz$ interpreted as $\{0\}$,
    \item $\cz^*$ interpreted as $\{n-1\}$,
    \item $\ps$ interpreted according to \cref{Def-discreteassociatedfunc},
\end{enumerate}
is definitionally equivalent (\cref{Definition-DefEquiv}) to $\MSO(\ul{n})$.
We will abuse notation and write $\MSO(\ul{n})$ both the $\langMO$-structure and the $\langMFIN$-structure on $\pow(\ul{n})$. 
As this definitional equivalence is uniform, meaning the same formulas can be used to translate between $\langMO$ and $\langMFIN$ for each $n \in \bb{N}$, we get an $\langMFIN$-theory which is definitionally equivalent to $T_{\MSO(\Fin)}$. 
We will again abuse notation and write $T_{\MSO(\Fin)}$ for this $\langMFIN$-theory, which is the shared theory and of the $\MSO(\ul{n})$ viewed as $\langMFIN$-structures.
\end{definition}

The following result serves as a foundation for the proof of the main theorem of this note \cref{Theorem-WIModelComplete}.
The proof is a straightforward application of automata normal form, directly inspired by the paper \cite{GvG}, but is omitted here.
For the proof refer to \cite{DLPhD}.

\begin{theorem}\label{Theorem-TMFinModelComplete}
The $\langMFIN$-theory $T_{\MSO(\Fin)}$ is model-complete.
\end{theorem}

\newsubsection{The weak monadic second order theory of a dense linear order}

Fix a dense linear order $\lcro$, with a left endpoint but no right endpoint.
By the fact that any completion of the theory of dense linear orders is $\aleph_0$-categorical, and using Proposition 2.7 from \cite{TresslTop}, it follows that our results do not depend the particular choice of $\lcro$.
We will in \cref{Subsection-ParametersGP} take $\lcro$ to be countable, this will allow us to streamline certain proofs, however it is in no way essential.

\begin{definition}\label{Definition-LWI}
We write $0$ for the left endpoint of $\lcro$, i.e. the smallest element with respect to the ordering.
We introduce a signature,
\[
\langWI = \{\cup,\cap,\bot,\cz,\min,\max,\ips\}.\footnote{This comprises in order, two binary function symbols, two constants, two unary function symbols, and a binary function symbol.}
\]
The $\langWI$-structure with,
\begin{enumerate}
\item universe $\pow_{\fin}(\lcro)$, the collection of finite subsets of $\lcro$,
\item $\cup,\cap$ interpreted as the operations of union and intersection, 
\item $\bot$ interpreted as the empty set,
\item $\cz$ interpreted as $\{0\}$,
\item $\min$ and $\max$ interpreted as the operations taking a non-empty finite set to the singleton containing its minimum and maximum respectively, with respect to the ordering of $\lcro$ (and both fixing $\bot$),
\item $\ips$ being the binary function given by,
\[
\ips(A,B) = \{i \in A:\suf_A(i) \in B\},
\]
i.e. $\ips$ is a single binary function which encodes the (preimage map associated to) the family of successor functions on finite subsets of $\lcro$,
\end{enumerate}
is definitionally equivalent to the $\langMO$-structure $\WSO(\lcro)$.
We will abuse notation and write $\WSO(\lcro)$ for both the $\langWI$-structure just described, and the $\langMO$-structure $\WSO(\lcro)$. 
\end{definition}

We will prove that as an $\langWI$-structure $\WSO(\lcro)$ is model-complete in \cref{Section_WIModelComplete}.

\newsection{Interpretations}

The main purpose of this section is to make explicit a connection between $\WSO(\lcro)$ and $T_{\MSO(\Fin)}$ (\cref{Proposition-WIinterpretsTMFin}), and then show that this connection behaves particularly nicely with respect to $\langWI$ and $\langMFIN$-embeddings (\cref{Proposition-EmbeddingRestriction}).
We also show that $\WSO(\lcro)$ interprets itself via restriction down to any left-closed right-open interval.

\newsubsection{\texorpdfstring{Interpreting $T_{\MSO(\Fin)}$ in $\WSO(\lcro)$}{}}

Let $A \in \WSO(\lcro)$ be a finite subset of $\lcro$.
If  $A$ has cardinality $n \in \bb{N}$, then viewed as a  sub-order of $\lcro$ we get that $A$ is isomorphic to $\ul{n}$. 
Hence $\WSO(\lcro) \upharpoonright A \models T_{\MSO(\Fin)}$ for each $A \in \WSO(\lcro)$, as in fact $\WSO(\lcro) \upharpoonright A \cong \MSO(\ul{n})$ for some $n \in \bb{N}$.

For each $n \in \bb{N}$ it is clear that we can find some $A \in \WSO(\lcro)$ of size $n$.
As such the pseudofinite monadic second order theory of linear order, $T_{\MSO(\Fin)}$, is the shared theory of restrictions to elements in $\WSO(\lcro)$.
For each $\langMO$-sentence $\phi$,
\[
T_{\MSO(\Fin)} \models \phi \Longleftrightarrow \WSO(\lcro) \models \forall X (\phi \downharpoonright X).
\]
From this precise statement it becomes clear that this phenomenon only depends on the theory of $\WSO(\lcro)$ and hence if $\str{M} \equiv \WSO(\lcro)$ then $\str{M} \upharpoonright A \models T_{\MSO(\Fin)}$ for each $A \in \str{M}$ (where $\str{M} \upharpoonright A$ is defined in the obvious way).
However, $\str{M} \upharpoonright A$ need not be isomorphic to $\MSO(\ul{n})$, rather $\str{M} \upharpoonright A$ may be a non-standard (i.e. non-finite) model of $T_{\MSO(\Fin)}$.
More detail on what these models look like can be found in \cite{DLPhD}.

The following proposition is just a restatement of what we have just said, for the purposes of referencing later on.

\begin{proposition}\label{Proposition-WIinterpretsTMFin}
Let $\str{M}$ be a model of $\Th(\WSO(\lcro))$. 
Then for each $A \in \str{M}$ the restriction of $\str{M}$ to $A$, i.e. $\str{M} \upharpoonright A$, is a model of $T_{\MSO(\Fin)}$.
\end{proposition}

\newsubsection{\texorpdfstring{Interpreting $\WSO(\lcro)$ within itself}{}}\label{Subsection-TWISelfInterpret}

Working in $\lcro$, it is clear that if we take an interval of the form $[i,j)$ with $i \in \lcro$ and $j \in \lcro \cup \{+\infty\}$ (of course with $i < j$), we get a sub-order which is again a dense linear order with left endpoint but no right-endpoint.
For such $i$ and $j$ we therefore have $\WSO(\lcro) \equiv \WSO(\lcro) \upharpoonright [i,j)$ (see \cref{Definition-ResInterval} for the definition of the latter).

For each $\langMO$-sentence $\phi$ we therefore have,
\begin{align*}
\WSO(\lcro) \models \phi &\Longleftrightarrow \WSO(\lcro) \models \exists i,j (\chi(i,j) \wedge \phi \downharpoonright [i,j)),\\
&\Longleftrightarrow \WSO(\lcro) \models \forall i,j (\chi(i,j) \rightarrow \phi \downharpoonright [i,j)),
\end{align*}
where $\chi(i,j)$ is a formula capturing the basic conditions required.

\begin{proposition}
Let $\str{M}$ be a model of $\Th(\WSO(\lcro))$.
For each pair of atoms $i,j \in \At(\str{M})$ such that $i < j$,
\[
\str{M} \equiv \str{M} \upharpoonright [i,j).
\]
\end{proposition}

\newsubsection{\texorpdfstring{Embeddings and the interpretation of $T_{\MSO(\Fin)}$ in $\WSO(\lcro)$}{}}

In both of the previous subsections we worked with $\langMO$-structures. 
Both interpretations can easily be translated to make use of the signatures $\langMFIN$ and $\langWI$ introduced earlier.
This is because of the fact that the structures we are interested in for both of these signatures are definitionally equivalent to $\langMO$-structures.

For each $\langWI$ structure $\str{M}$ which is elementarily equivalent to $\WSO(\lcro)$, and each $A \in \str{M}$, we can make sense of $\str{M} \upharpoonright A$ as an $\langMFIN$-structure.
Moreover we have that $\str{M} \upharpoonright A \models T_{\MSO(\Fin)}$, viewing $T_{\MSO(\Fin)}$ as an $\langMFIN$-theory.

The following result which plays a crucial role in the proof of the model-completeness of $\WSO(\lcro)$, and is the basis of the choice of signature $\langWI$.\footnote{By which I mean, $\langWI$ was engineered precisely to yield the result.}

\begin{proposition}\label{Proposition-EmbeddingRestriction}
Let $\str{M} \subseteq \str{N}$ be an inclusion, as $\langWI$-structures, of models of $\Th(\WSO(\lcro))$ and let $A$ be an element of $\str{M}$.
Then the inclusion restricts to an inclusion $\str{M} \upharpoonright A \subseteq \str{N} \upharpoonright A$ of $\langMFIN$-structures. 
\end{proposition}
\begin{proof}
Recall that the signatures we work in are $\langMFIN = \{\cup,\cap,\bot,\cz,\cz^*,\ps\}$ and $\langWI = \{\cup,\cap,\bot,\cz,\min,\max,\ips\}$. 
Firstly we can observe that the inclusion $\str{M} \subseteq \str{N}$ (of $\langWI$-structures) does restrict to an inclusion as sets $\str{M} \upharpoonright A \subseteq \str{N} \upharpoonright A$. 
This is because for each $B \in \str{M} \upharpoonright A$ we have $\str{M} \models B \cap A = B$, and hence $\str{N} \models B \cap A = B$, giving $B \in \str{N} \upharpoonright A$.

It is sufficient to check that for each unnested atomic $\langMFIN$-formula $\phi(\bar{X})$ and each $\bar{B} \in \str{M} \upharpoonright A$,
\[
\str{M} \upharpoonright A \models \phi(\bar{B}) \Leftrightarrow \str{N} \upharpoonright A \models \phi(\bar{B}).
\]
This can be done by cases, going through the finitely many symbols from the signature $\langMFIN$. 
We will demonstrate with one example, which is the most involved. 
Consider the $\langMFIN$-formula $\ps(X_1)=X_2$. 
For each $B_1,B_2 \in \str{M} \upharpoonright A$,
\begin{align*}
\str{M} \upharpoonright A \models \ps(B_1) = B_2 &\Longleftrightarrow \str{M} \models \ips(B_1,A) = B_2,\\
&\Longleftrightarrow \str{N} \models \ips(B_1,A) = B_2,\\
&\Longleftrightarrow \str{N} \upharpoonright A \models \ps(B_1) = B_2.
\end{align*}
The middle bi-implication is a direct consequence of the assumption that $\str{M} \subseteq \str{N}$ is an inclusion of $\langWI$-structures.
All remaining cases can be dealt with in a similar fashion. 
\end{proof}

\newsection{Feferman-Vaught}\label{Section_FefVau}

In \cref{Proposition-WIinterpretsTMFin} we saw that through restriction to an element, $\Th(\WSO(\lcro))$ interprets $T_{\MSO(\Fin)}$.
The way that this is actually used is in taking $\langMO$-formulae $\phi(\bar{X})$ and giving $\langMO$-formulae $\psi(\bar{X},Y)$ such that $\psi$ \emph{says} ``$\phi$ is true in the restriction to $Y$''.
In this section we will establish a result which can be construed as a kind of converse. 
Taking an $\langMO$-formula $\phi(\bar{X})$ we will produce an $\langMO$-formula $\psi(\bar{X})$.
For each $\langMO$-structure $\str{M}$ which is a model of $\Th(\WSO(\lcro))$ and each tuple $\bar{A}$ in $\str{M}$, the idea will be that in the restriction $\str{M} \upharpoonright \funky{\bar{A}}$, $\psi(\bar{A})$ \emph{says} that $\phi(\bar{A})$ holds in $\str{M}$.

The way we will establish this result is to use the Feferman-Vaught machinery of generalised products.
In our case we will only require the special case of generalised powers. 
In particular, we will show that for any tuple $\bar{A}$ in $\WSO(\lcro)$, the pointed structure $(\WSO(\lcro),\bar{A})$ is equivalent to a generalised power of the structure $\WSO(\lcro)$ over the pointed structure $(\WSO(\lcro) \upharpoonright \funky{\bar{A}},\bar{A})$.

\newsubsection{Definitions and setup}

We will first briefly recall the essential notions from \cite{FefVau}, restricting attention to the special cases which we make use of.

\begin{definition}\label{Definition-AlgebraSets} \
\begin{enumerate}
    \item An \textbf{algebra of sets} is a structure of the form $\str{I} = (\pow(I),\subseteq,\ldots)$ for some set $I$, in a \emph{relational} signature $\lang{\ind}{}$ extending $\{\subseteq\}$, in which $\subseteq$ is inclusion of sets.
    When we want to highlight the choice of $\lang{\ind}{}$ we will talk of an $\lang{\ind}{}$-algebra of sets.
    
    \item For an $\lang{}{}$-structure $\str{M}$ we will write $\str{M}^I$ for the collection of functions from $I$ to $\str{M}$.
    For an $\lang{}{}$-formula $\theta(\bar{x})$ and a tuple $\bar{f} \in (\str{M}^{I})^{\bar{x}}$, we write $\dVert{\theta(\bar{f})}$ for the set,
    \[
    \{i \in I: \str{M} \models \theta(\bar{f}(i))\},
    \]
    which we call the \textbf{support} of $\theta$ at $\bar{f}$.
    Note that $\dVert{\theta(\bar{f})}$ is in $\str{I}$.
    
    \item If $\lang{}{}$ is a signature, then an \textbf{acceptable sequence} of $\lang{}{}$ over $\lang{\ind}{}$ is a tuple,
    \[
    \zeta(\bar{x}) = (\Phi(Y_1,\ldots,Y_m),\theta_1(\bar{x}),\ldots,\theta_m(\bar{x})),
    \]
    where $\Phi$ is an $\lang{\ind}{}$-formula, and $\theta_1,\ldots,\theta_m$ are $\lang{}{}$-formulas, for some $m \in \bb{N}$.
    
    \item Let $\lang{\pi}{}$ be the signature given by taking an $l(\bar{x})$-ary relation symbol $R_{\zeta}$ for each acceptable sequence $\zeta(\bar{x})$ of $\lang{}{}$ over $\lang{\ind}{}$.
    The \textbf{generalised power} of $\str{M}$ by $\str{I}$, denoted by $\str{M}^{\str{I}}$, is an $\lang{\pi}{}$-structure with universe $\str{M}^I$. 
    For each acceptable sequence,
    \[
    \zeta(\bar{x}) = (\Phi(Y_1,\ldots,Y_m),\theta_1(\bar{x}),\ldots,\theta_m(\bar{x})),
    \]
    the relation symbol $R_{\zeta}$ is interpreted as,
    \[
    \{\bar{f} \in \str{M}^I: \str{I} \models \Phi(\dVert{\theta_1(\bar{f})},\ldots,\dVert{\theta_m(\bar{f})})\}.
    \]
\end{enumerate}
\end{definition}

\begin{theorem}[Special case of \cite{FefVau} Theorem 3.1 pp 5]\label{Theorem-FefVauPowers}
Let $\lang{}{}$ be a first-order signature, and let $\lang{\ind}{}$ be a relational signature extending $\{\subseteq\}$.
Then let $\lang{\pi}{}$ be the associated signature for generalised powers of $\lang{}{}$-structures by $\lang{\ind}{}$-algebras of sets.

For each $\lang{\pi}{}$-formula $\phi(\bar{x})$, there is an atomic $\lang{\pi}{}$-formula $\phi^*(\bar{x})$, i.e. one of the form $R_{\zeta}(\bar{x})$ for some acceptable sequence $\zeta(\bar{x})$, such that for each $\lang{}{}$-structure $\str{M}$ and $\lang{\ind}{}$-algebra of sets $\str{I}$,
\[
\str{M}^{\str{I}} \models \forall \bar{x} (\phi(\bar{x})  \leftrightarrow \phi^*(\bar{x})).
\]

In particular, if $\phi$ is an $\lang{\pi}{}$-sentence then $\phi^*$ is of the form $R_{\zeta}$ where $\zeta$ is an acceptable sequence,
\[
(\Phi(Y_1,\ldots,Y_m),\theta_1,\ldots,\theta_m)
\]
where $\theta_1,\ldots,\theta_m$ are $\lang{}{}$-sentences. 
\end{theorem}

\begin{proposition}\label{Proposition-GenPowerSentence}
Let $\lang{}{}$, $\lang{\ind}{}$, and $\lang{\pi}{}$ be given as in \cref{Theorem-FefVauPowers}.
Then for each $\lang{\pi}{}$-sentence $\phi$ and $\lang{}{}$-structure $\str{M}$, there is an $\lang{\ind}{}$-sentence $\Phi$ such that for each $\lang{\ind}{}$-algebra of sets $\str{I}$, 
\[
\str{M}^{\str{I}} \models \phi \Longleftrightarrow \str{I} \models \Phi.
\]
\end{proposition}
\begin{proof}
By \cref{Theorem-FefVauPowers} there is an acceptable sequence,
\[
\zeta = (\Phi(Y_1,\ldots,Y_m),\theta_1,\ldots,\theta_m),
\]
where $\theta_1,\ldots,\theta_m$ are sentences, such that $\str{M}^{\str{I}} \models \phi$ if and only if,
\[
\str{I} \models \Phi(\dVert{\theta_1},\ldots,\dVert{\theta_m}).
\]
But now that $\str{M}$ is fixed, each of the $\dVert{\theta_i}$ are either $\bot$ or $\top$, both of which are definable elements in $\str{I}$.
Abusing notation and writing $\Phi$ for the sentence obtained by modifying the previous $\Phi$ to refer to these definable elements, we are done.
\end{proof}

\newsubsection{\texorpdfstring{The case of parameters over $\WSO(\lcro)$}{The case of parameters over W(I)}}\label{Subsection-ParametersGP}

Let $\bar{A}$ be a tuple from $\WSO(\lcro)$.
For the sake of convenience, we will think of $(\WSO(\lcro),\bar{A})$ as a pointed $\langMO$-structure. 
That is we take the $\langMO$-structure $\WSO(\lcro)$ and then expand the signature by constants naming each entry from the tuple $\bar{A}$. 

\begin{definition}
We define a map $\funky{\ }: \bigcup_{n \in \bb{N}} \WSO(\lcro)^n \rightarrow \WSO(\lcro)$ as follows,
\[
\funky{\bar{A}} = \cz \cup \bigcup \bar{A}.
\]
So $\funky{\bar{A}}$ is the union of each of the entries in $\bar{A}$, together with $\{0\}$. 
\end{definition}

Note that for each $n \in \bb{N}$, the restriction of $\funky{\ }$ to tuples of length $n$ is clearly definable in $\WSO(\lcro)$ (it is given by a term in the $\langWI$ version of $\WSO(\lcro)$, which is definitionally equivalent to the $\langMO$ version).

The usefulness of $\funky{\bar{A}}$ comes from the fact that when we restrict $\bar{A}$ to each of the intervals $[i,j)$ where $i$ and $j$ are elements of $\funky{\bar{A}}$, we get a tuple of definable elements.
This is because each $A$ from $\bar{A}$ restricts down to either $\bot$ (if $i \notin A$) or $\cz$ (if $i \in A$).
Note that here we are reinterpreting the constant symbol $\cz$ in the structure $\WSO(\lcro) \upharpoonright [i,j)$, it is interpreted as the smallest atom with respect to the linear ordering of the interval which in this case is $\{i\}$.

\begin{definition}
Let $\bar{A}$ be a tuple in $\WSO(\lcro)$. 
For the sake of readability we will write $\WSO\funky{\bar{A}}$ for $\WSO(\lcro) \upharpoonright \funky{\bar{A}}$. 
Note that $\bar{A}$ itself can be thought of as a tuple in $\WSO\funky{\bar{A}}$.

The $\langMO$-structure $\WSO\funky{\bar{A}}$ is a model of $T_{\MSO(\Fin)}$.
It is an algebra of sets (see \cref{Definition-AlgebraSets}), and remains so when we name the tuple $\bar{A}$ as constants. 
We will write $\WSO(\lcro)^{(\bar{A})}$ for the generalised power of the $\langMO$-structure $\WSO(\lcro)$ by the algebra of sets $(\WSO\funky{\bar{A}},\bar{A})$.
\end{definition}

Now let $\lcro$ be countable. 
As the theory of dense linear orders with left-endpoint but no right endpoint is $\aleph_0$-categorical, we can simply choose to take $\lcro = [0,1)_{\bb{Q}}$ without loss of generality.
For any $i<j$ in $\lcro$ it is clear that $\WSO(\lcro) \upharpoonright [i,j)$ (see \cref{Subsection-TWISelfInterpret}) is isomorphic to $\WSO(\lcro)$ (and similarly if $j = +\infty$).

\begin{definition}
(We require that $\lcro$ is countable.) 
Let $\bar{A}$ be a tuple in $\WSO(\lcro)$.
For each $i \in \funky{\bar{A}}$ let $h_i:\WSO(\lcro) \rightarrow \WSO(\lcro) \upharpoonright [i,j)$ be an isomorphism witnessing the fact that $\WSO(\lcro) \cong \WSO(\lcro) \upharpoonright [i,j)$ where $j$ is the successor of $i$ in $\funky{\bar{A}}$ (or $j = +\infty$ for $i = \max(\funky{\bar{A}})$).
Then we will define $g_{\bar{A}}:\WSO(\lcro)^{(\bar{A})} \rightarrow \WSO(\lcro)$ to be the function given by,
\[
g_{\bar{A}}(f) = \bigcup_{i \in \funky{\bar{A}}} h_i(f(i)).
\]
\end{definition}

For each tuple $\bar{A}$ in $\WSO(\lcro)$ the function $g_{\bar{A}}$ is a bijection from $\WSO(\lcro)^{(\bar{A})}$ to $\WSO(\lcro)$.

\begin{proposition}
(We require that $\lcro$ is countable.) 
For each tuple $\bar{A}$ in $\WSO(\lcro)$ the generalised power $\WSO(\lcro)^{(\bar{A})}$ of $\WSO(\lcro)$ by $(\WSO\funky{\bar{A}},\bar{A})$ is definitionally equivalent to the structure $(\WSO(\lcro),\bar{A})$, after identifying their universes using the function $g_{\bar{A}}$. 
Moreover this equivalence is uniform in the sense that the same formulas can be used to establish the equivalence for any choice of tuple $\bar{A}$ (of suitable length).
\end{proposition}
\begin{proof}
We first prove that both $\subseteq$ and $\exle$ are definable in $\WSO(\lcro)^{(\bar{A})}$ by giving acceptable sequences. 
For $\subseteq$ we can take $\zeta = (\Phi,\theta_1)$ with $\theta_1(x_1,x_2)$ being $x_1 \subseteq x_2$ and $\Phi$ being $Y_1 = \top$ (this is clearly definable in $\langMO$).
For $\exle$ we take $\zeta = (\Phi,\theta_1,\theta_2,\theta_3)$ with $\theta_1(x_1,x_2)$ being $x_1 \exle x_2$, $\theta_2(x_1,x_2)$ being $x_1 \neq \bot$, and $\theta_3(x_1,x_2)$ being $x_2 \neq \bot$, and,
\[
\Phi: Y_1 \neq \bot \vee Y_2 \exle Y_3.
\]
Now we show that each constant standing for an entry $A \in \bar{A}$ is definable using an acceptable sequence.
For this we can take $\zeta = (\Phi,\theta_1,\theta_2)$ with $\theta_1(x_1)$ being $x_1 = \cz$, $\theta_2(x_1)$ being $x_1 = \bot$, and,
\[
\Phi: Y_1 = A \wedge Y_2 = A^c.
\]
We have used many short hands, all of which can be worked into $\langMO$-formulas explicitly by the reader. 

In the other direction, take an acceptable sequence,
\[
\zeta = (\Phi(Y_1,\ldots,Y_m),\theta_1,\ldots,\theta_m).
\]
If we take a tuple $\bar{B}$ from $\WSO(\lcro)$ then using the techniques of restriction and relativisation (and a lot of patience) it can be seen that each of the supports,
\[
\dVert{\theta_1(g_{\bar{A}}^{-1}(\bar{B}))},\ldots,\dVert{\theta_m(g_{\bar{A}}^{-1}(\bar{B}))},
\]
are definable in $(\WSO(\lcro),\bar{A})$ over $\bar{B}$ (viewing each as a subset of $\funky{A}$ living inside $\WSO(\lcro)$).
Moreover $\WSO\funky{\bar{A}}$ is clearly interpretable in $(\WSO(\lcro),\bar{A})$, and in such a way that the supports of $\theta_1,\ldots,\theta_m$ can be fed into $\WSO\funky{\bar{A}}$.
Therefore $R_{\zeta}$ is definable in $(\WSO(\lcro),\bar{A})$.
\end{proof}

The following proposition is then just a matter of switching our perspective back to thinking of $\bar{A}$ as a tuple of parameters rather than as a tuple of constants baked into the signature.

\begin{proposition}\label{Proposition-WIParametersGeneralisedProduct}
For each $\langMO$-formula $\phi(\bar{X})$ there is an $\langMO$-formula $\psi(\bar{X})$ such that for each tuple $\bar{A}$ in $\WSO(\lcro)$,
\[
\WSO(\lcro) \models \phi(\bar{A}) \Longleftrightarrow \WSO\funky{\bar{A}} \models \psi(\bar{A}).
\]
\end{proposition}
\begin{proof}
We have that $(\WSO(\lcro),\bar{A})$ is definitionally equivalent to $\WSO(\lcro)^{(\bar{A})}$ (and that this equivalence is uniform).
For each $\langMO$-formula $\phi(\bar{X})$, there is clearly an $\langMO \cup \{\bar{A}\}$-sentence $\phi'$ such that,
\[
\WSO(\lcro) \models \phi(\bar{A}) \Longleftrightarrow (\WSO(\lcro),\bar{A}) \models \phi'.
\]
Then by \cref{Proposition-GenPowerSentence} there is an $\langMO \cup \{\bar{A}\}$-sentence $\psi'$ such that,
\[
(\WSO(\lcro),\bar{A}) \models \phi' \Longleftrightarrow (\WSO\funky{\bar{A}},\bar{A}) \models \psi'.
\]
But now there is clearly an $\langMO$-formula $\psi(\bar{X})$ such that,
\[
(\WSO\funky{\bar{A}},\bar{A}) \models \psi' \Longleftrightarrow \WSO\funky{\bar{A}} \models \psi(\bar{A}).
\]
So we are done.
\end{proof}

\begin{proposition}\label{Proposition-ParametersGeneralisedProduct}
For each $\langWI$-formula $\phi(\bar{X})$ there is an $\langMFIN$-formula $\psi(\bar{X})$ such that for each,
\begin{enumerate}
    \item $\langWI$-structure $\str{M}$ which is a model of $\Th(\WSO(\lcro))$ and,
    \item tuple $\bar{A}$ in $\str{M}$,
\end{enumerate}
the following equivalence holds,
\[
\str{M} \models \phi(\bar{A}) \Longleftrightarrow \str{M} \upharpoonright \funky{\bar{A}} \models \psi(\bar{A}).
\]
\end{proposition}
\begin{proof}
This follows surprisingly easily from \cref{Proposition-WIParametersGeneralisedProduct}.
Let $\phi(\bar{X})$ be an $\langWI$-formula. 
This can be replaced with an equivalent $\langMO$-formula.
Now there is an $\langMO$-formula $\psi(\bar{X})$ such that for each tuple $\bar{A}$ in $\WSO(\lcro)$,
\[
\WSO(\lcro) \models \phi(\bar{A}) \Longleftrightarrow \WSO\funky{\bar{A}} \models \psi(\bar{A}).
\]
Note that this equivalence is actually first order expressible.
By this we mean that we can produce an $\langMO$-formula $\Omega_{\phi,\psi}(\bar{X})$ \emph{saying} that `$\phi$ holds at $\bar{X}$ if and only if $\psi$ holds at $\bar{X}$ in the restriction to $\funky{\bar{X}}$'.
This follows from restriction to an element being an interpretation and the definability of the function $\funky{\ }$.

But now $\WSO(\lcro) \models \forall \bar{X} (\Omega_{\phi,\psi}(\bar{X}))$ for this choice of $\phi$ and $\psi$, hence from our assumption that $\str{M} \models \Th(\WSO(\lcro))$ it follows that $\str{M} \models \forall \bar{X} (\Omega_{\phi,\psi}(\bar{X}))$.
Unpacking what $\Omega_{\phi,\psi}(\bar{X})$ \emph{says}, for each tuple $\bar{A}$ in $\str{M}$ we have that, 
\[
\str{M} \models \phi(\bar{A}) \Longleftrightarrow \str{M} \upharpoonright \funky{\bar{A}} \models \psi(\bar{A}).
\]
This completes the proof.
\end{proof}

\newsection{\texorpdfstring{Proof that $\WSO(\lcro)$ is model-complete}{Proof that W(I) is model-complete}}\label{Section_WIModelComplete}

In \cref{Subsection-WIModelComplete} we will prove our main result \cref{Theorem-WIModelComplete}.
In \cref{Section-transfer} we will need a slightly stronger result, replacing model-complete with positive model-complete in \cref{Theorem-WIModelComplete}.
We prove this stronger result in \cref{Subsection-WIPositiveModelComplete}.

\newsubsection{\texorpdfstring{Proof that $\WSO(\lcro)$ is model-complete}{Proof that W(I) is model-complete}}\label{Subsection-WIModelComplete}

\begin{theorem}\label{Theorem-WIModelComplete}
$\WSO(\lcro)$ is model-complete in the signature $\langWI$.
\end{theorem}
\begin{proof}
We will show that every embedding of models of $\Th(\WSO(\lcro))$ is elementary.
Fix an $\langWI$-formula $\phi(\bar{X})$. 
Moreover let,
\begin{enumerate}
\item $\str{M} \subseteq \str{N}$ be an extension of models of $\Th(\WSO(\lcro))$ (as $\langWI$-structures) and,
\item $\bar{A}$ be a tuple in $\str{M}$.
\end{enumerate}
We will show that,
\[\label{elemeq}
\str{M} \models \phi(\bar{A}) \Longleftrightarrow \str{N} \models \phi(\bar{A}). \tag{*}
\]
By \cref{Proposition-ParametersGeneralisedProduct}, there is an $\langMFIN$-formula $\psi(\bar{X})$ (depending only on $\phi$, not on $\str{M}$, nor $\str{N}$, nor $\bar{A}$), such that,
\[\label{innereq}
\str{M} \models \phi(\bar{A}) \Longleftrightarrow \str{M} \upharpoonright \funky{\bar{A}} \models \psi(\bar{A}), \tag{$\dagger$}
\]
and,
\[\label{innereq2}
\str{N} \models \phi(\bar{A}) \Longleftrightarrow \str{N} \upharpoonright \funky{\bar{A}} \models \psi(\bar{A}). \tag{$\dagger'$}
\]
By \cref{Proposition-EmbeddingRestriction}, the inclusion $\str{M} \subseteq \str{N}$ as $\langWI$-structures restricts to the inclusion $\str{M} \upharpoonright \funky{\bar{A}} \subseteq \str{N} \upharpoonright \funky{\bar{A}}$ as $\langMFIN$-structures.
By \cref{Proposition-WIinterpretsTMFin} both $\str{M} \upharpoonright \funky{\bar{A}}$ and $\str{N} \upharpoonright \funky{\bar{A}}$ are models of $T_{\MSO(\Fin)}$. 
As an $\langMFIN$-theory, $T_{\MSO(\Fin)}$ is model-complete (\cref{Theorem-TMFinModelComplete}).
Therefore $\str{M} \upharpoonright \funky{\bar{A}} \prec \str{N} \upharpoonright \funky{\bar{A}}$ and in particular,
\[\label{reseleq}
\str{M} \upharpoonright \funky{\bar{A}} \models \psi(\bar{A}) \Longleftrightarrow \str{N} \upharpoonright \funky{\bar{A}} \models \psi(\bar{A}). \tag{$\ddagger$}
\]
Now collecting everything together we have,
\begin{align*}
\str{M} \models \phi(\bar{A}) &\Longleftrightarrow \str{M} \upharpoonright \funky{\bar{A}} \models \psi(\bar{A}), & \text{(by \eqref{innereq})}\\
&\Longleftrightarrow \str{N} \upharpoonright \funky{\bar{A}} \models \psi(\bar{A}), & \text{(by \eqref{reseleq})}\\
&\Longleftrightarrow \str{N} \models \phi(\bar{A}). & \text{(by \eqref{innereq2})} 
\end{align*}
So we have established \eqref{elemeq} as required.
\end{proof}

\newsubsection{\texorpdfstring{$\WSO(\lcro)$ is positive model-complete}{W(I) is positive model-complete}}\label{Subsection-WIPositiveModelComplete}

With a small amount of effort we can improve \cref{Theorem-WIModelComplete} and show that $\WSO(\lcro)$ is positive model-complete, meaning that every $\langWI$-formula is equivalent over $\WSO(\lcro)$ to a positive existential $\langWI$-formula.
We need a couple of preparatory lemmas.

\begin{lemma}\label{Lemma-notbotelim}
The following are equivalent for each $A \in \pow_{\fin}(\lcro)$,
\begin{enumerate}
\item $A \neq \bot$,
\item $A = \cz$ or $\cz \subseteq \ips(A \cup \cz,A)$.
\end{enumerate}
\end{lemma}
\begin{proof}
$(1) \Rightarrow (2)$:
If $A \neq \bot$ and $A \neq \cz$ then it is immediate that $A \wo \cz$ is not $\bot$.
This gives us that $\suf_{A \cup \cz}(0)$ is defined (i.e. $0 \in \lcro$ has a successor in $\cz \cup A$).
Then $\suf_{A \cup \cz}(0) \in A$ which by definition gives us $\cz \subseteq \ips(A \cup \cz,A)$.

$(2) \Rightarrow (1)$:
Suppose $A = \bot$. 
From this it follows immediately that $A \neq \cz$.
Moreover $\ips(A \cup \cz,A) = \ips(\cz,\bot)$ which is $\bot$, hence $\cz \not\subseteq \ips(A \cup \cz,A)$.
\end{proof}

\begin{definition}
Recall that for sets $A$ and $B$, the \textbf{relative complement} of $B$ inside $A$, denoted by $A \wo B$, is the set $\{i \in A: i \notin B\}$. 
\end{definition}

\begin{lemma}\label{Prop-WIqffelimneg}
Every quantifier-free $\langWI$-formula is equivalent to a positive (i.e. negation free) existential $\langWI$-formula over the the theory of $\WSO(\lcro)$. 
\end{lemma}
\begin{proof}
First note that for each $A,B \in \pow_{\fin}(\lcro)$, $A \wo B = C$ if and only if,
\[ 
(A \cap  B) \cup C = A \text{ and }B \cap C = \bot.
\]
Therefore every positive existential $(\langWI \cup \{\wo\})$-formula is equivalent to a positive existential $\langWI$-formula (working in $\WSO(\lcro)$, with $\wo$ a binary function symbol interpreted as relative complement).
So it is enough to check that every quantifier-free $\langWI$-formula is equivalent to a positive existential $(\langWI \cup \{\wo\})$-formula.
It is enough to check that the negation of an atomic $(\langWI \cup \{\wo\})$-formula is equivalent to a positive existential $(\langWI \cup \{\wo\})$-formula in $\WSO(\lcro)$.
Atomic formulas are all of the form $q_1(\bar{X}) = q_2(\bar{X})$ for $\langWI$-terms $q_1,q_2$.
Let $\Delta(q_1(\bar{X}),q_2(\bar{X}))$ be shorthand for $(q_1(\bar{X}) \wo q_2(\bar{X})) \cup (q_2(\bar{X}) \wo q_1(\bar{X}))$, i.e. the symmetric difference of $q_1$ and $q_2$.
The formula $\neg (q_1(\bar{X}) = q_2(\bar{X}))$ is equivalent over $\WSO(\lcro)$ to,
\[
\exists Y (Y \neq \bot \wedge Y \subseteq \Delta(q_1(\bar{X}),q_2(\bar{X}))),
\]
but using \cref{Lemma-notbotelim} this in turn is equivalent over $\WSO(\lcro)$ to,
\[
\exists Y ((Y = \cz \vee \cz \subseteq \ips(Y \cup \cz,Y)) \text{ and }Y \subseteq \Delta(q_1(\bar{X}),q_2(\bar{X}))),
\]
which is positive existential as required.
\end{proof}

Combining \cref{Theorem-WIModelComplete,Prop-WIqffelimneg} yields the following slightly stronger result, which we will use in \cref{Section-transfer} to obtain another model-completeness result. 

\begin{proposition}\label{Corollary-WIPositiveMC}
$\WSO(\lcro)$ is positive model-complete in the signature $\langWI$.
\end{proposition}
\begin{proof}
By \cref{Theorem-WIModelComplete} every $\langWI$-formula is equivalent to an existential $\langWI$-formula over $\WSO(\lcro)$.
By \cref{Prop-WIqffelimneg}, we can substitute the quantifier free part of this existential formula with a positive existential formula and get an equivalent formula over $\WSO(\lcro)$.
This yields a positive existential formula, so we are done. 
\end{proof}

\newsection{\texorpdfstring{Model-completeness of $\LCI(\lcro)$}{Model-completeness of L(I)}}\label{Section-transfer}

In this section we will give a signature ($\langLI$, see \cref{Definition-LI}) in which the lattice $(\LCI(\lcro),\subseteq)$ of finite unions of closed intervals is model-complete.
Much like the model-completeness result for $\WSO(\lcro)$ utilised the model-completeness result for $T_{\MSO(\Fin)}$, the proof that $\LCI(\lcro)$ is model-complete makes use of the model-completeness result for $\WSO(\lcro)$.

The material in this section is reproduced, with almost no changes, from the earlier note \cite{DeaconLIpreprint}.

\newsubsection{\texorpdfstring{$\LCI(\lcro)$}{L(I)}, finite unions of closed intervals of \texorpdfstring{$\lcro$}{I}}

\begin{definition}\label{Definition-LI}
We write $\langLI$ for the signature $\{\cup,\cap,\bot,\cz,\min,\max,l,r\}$.\footnote{This comprises in order, two binary function symbols, two constants, and four unary function symbols.}

$\LCI(\lcro)$ is the $\langLI$-structure with,
\begin{enumerate}
\item universe $\pow_{\fci}(\lcro)$, the collection of finite unions of closed intervals of $\lcro$,
\item $\cup,\cap$ interpreted as the operations of union and intersection, 
\item $\bot$ interpreted as the empty set,
\item $\cz$ interpreted as $\{0\}$,
\item $\min$ and $\max$ interpreted as the operations taking an element of $\pow_{\fci}(\lcro)$ to the singleton containing its minimum and maximum respectively, with respect to the ordering of $\lcro$ (we set $\min(\bot) = \max(\bot) = \bot$, and in the case $A \in \pow_{\fci}(\lcro)$ is unbounded we set $\max(A)=\bot$),
\item $l$ and $r$ are interpreted as the operations taking an element of $\pow_{\fci}(\lcro)$ to the set of its left and right endpoints respectively.
\end{enumerate}
\end{definition}

\begin{notation}\label{notationlr}
For $A \in \pow_{\fci}(\lcro)$ we will sometimes write $A_l$ and $A_r$ in place of $l(A)$ and $r(A)$ respectively. We will write $\bdr{A}$ as shorthand for $A_l \cup A_r$.
\end{notation}

\begin{lemma}\label{Lemma-boundeddef}
The bounded elements of $\pow_{\fci}(\lcro)$ form a definable subset in  $\LCI(\lcro)$, which we will denote by $\bdd$.
\end{lemma}
\begin{proof}
Under our interpretation, an element $A \in \pow_{\fci}(\lcro)$ is bounded if and only if $A = \bot$ or $\max(A) \neq \bot$. 
Conversely the unbounded elements are precisely those $A \in \pow_{\fci}(\lcro)$ for which $A \neq \bot$ and $\max(A) = \bot$.

Therefore $\bdd$ is in fact a quantifier-free definable set in $\LCI(\lcro)$. 
\end{proof}

\newsubsection{Interpreting \texorpdfstring{$\WSO(\lcro)$ in $\LCI(\lcro)$}{W(I) in L(I)}}

\begin{proposition}
The set $\pow_{\fin}(\lcro)$ is quantifier-free definable in $\LCI(\lcro)$.
\end{proposition}
\begin{proof}
The formula $l(X) = r(X)$ defines $\pow_{\fin}(\lcro)$ in $\LCI(\lcro)$.
\end{proof}

We will use this as the foundation for our interpretation of $\WSO(\lcro)$ in $\LCI(\lcro)$.

In the remainder of this section we will outline how to define the remaining $\langWI$-structure carried by $\WSO(\lcro)$ within $\LCI(\lcro)$ on the set $\pow_{\fin}(\lcro)$.
 
It will be important for us, when transferring model-completeness from $\WSO(\lcro)$ to $\LCI(\lcro)$ in \cref{Subsection-transfer}, that existential $\langLI$-formulae are used to do so.

\begin{remark}
If $\phi$ is an unnested atomic formulae in the signature $\langWI \cap \langLI$, i.e. $\{\cup,\cap,\bot,\cz,\min,\max\}$, we have that,
\[
\phi(\WSO(\lcro)) = \phi(\LCI(\lcro)) \cap \pow_{\fin}(\lcro).
\]
As such for unnested atomic formulae in this reduct, we need do nothing when giving the interpretation. 
Here we use unnested atomic $\langLI$-formulas, which a fortiori are existential. 
\end{remark}

All that is left is to produce an $\langLI$-formula which defines $\ips$ in $\LCI(\lcro)$.

\begin{remark}
Let $A,B \in \pow_{\fin}(\lcro)$. 
If $A = \bot$ or $B = \bot$ then $\ips(A,B) = \bot$.
Therefore in defining $\ips$ we can assume that $A,B \neq \bot$.

Moreover we have that $\ips(A,B) = \ips(A,B \cap A)$. 
From this it is clear that it is sufficient to define the relation $\ips(A,B) = C$ in the case where $B \subseteq A$. 
\end{remark}

\begin{proposition}
Let $A,B,C \in \pow_{\fin}(\lcro)$ with $\bot \subsetneq B \subseteq A$. Then $\ips(A,B) = C$ if and only if there exists $D \in \pow_{\fci}(\lcro)$ such that one of the following holds,
\begin{enumerate}
\item $\min(A) \subseteq B$, $l(D) = (B \wo \min(B)) \cup \cz$, $r(D) = C$, and $C \subseteq A \subseteq D$, or,
\item $\min(A) \not\subseteq B$, $l(D) = B \cup \cz$, $r(D) = C$, and $C \subseteq A \subseteq D$.
\end{enumerate} 
\end{proposition}
\begin{proof}
First suppose that $\ips(A,B)=C$, we will show that one of the two conditions must then hold.
Let $E$ be the union of all of the \emph{open} intervals of $\lcro$ of the form $(i,j)$ for some $i \in C$ and $j \in B$ such that $\suf_{A}(i) = j$. 
Now take $D = \lcro \wo E \in \pow_{\fci}(\lcro)$.

By definition we get that $C \subseteq A$. Then $A \subseteq D$ follows from the choice that $E$ be made up from intervals $(i,j)$ where $\suf_A(i) = j$.
For if $A \not\subseteq D$ then we get $k \in A$ such that for some $i,j$ with $\suf_A(i)=j$ we have $i < k < j$, a contradiction.

For our choice of $E$, it is easy to check that $l(E) = C$ and $r(E)$ is either,
\begin{enumerate}[(i)]
\item $B \wo \min(B)$ if $\min(A) \subseteq B$, or,
\item $B$ if $\min(A) \not\subseteq B$.
\end{enumerate}
Taking the complement, we interchange left and right endpoints, and introduce $0$ as a left endpoint.
This gives us precisely that $(1)$ or $(2)$ hold for our choice of $D$.

Now for the other direction, suppose that $D \in \pow_{\fci}(\lcro)$ exists such that $(2)$ holds. 
We want to show that $\ips(A,B) = C$. 
Let $i \in \ips(A,B)$, so $i \in A$ is such that $\suf_A(i) \in B \subseteq l(D)$. 
Now $i \notin C = r(D)$ implies that $i \in A \wo D$, contradicting our assumption that $A \subseteq D$.
So we have established that $\ips(A,B) \subseteq C$ follows from $(2)$.
Let $i \in C$, so by $(2)$ we have $i \in r(D)$. 
Suppose towards a contradiction that $\suf_A(i) \notin B$. 
Then moreover $\suf_i(A) \notin B \cup \cz = l(D)$.
This again gives us that $i \in A \wo D$, contradicting our assumption that $A \subseteq D$. 
Therefore we have shown that $\ips(A,B)=C$ follows from $(2)$.

We leave the checking of details in showing that $\ips(A,B)=C$ follows from $(1)$ to the reader.
\end{proof}

\begin{corollary}
There is an existential $\langLI$-formula $\phi(X,Y,Z)$ such that for all $A,B,C \in \pow_{\fin}(\lcro)$, $\LCI(\lcro) \models \phi(A,B,C)$ if and only if $\ips(A,B) = C$.
\end{corollary}

\begin{corollary}\label{LinterpretsW}
For each $\langWI$-formula $\phi(\bar{X})$ there is an $\langLI$-formula $\psi(\bar{Y})$ (with $\bar{X}$ and $\bar{Y}$ having the same length) such that for each $\bar{A} \in \pow_{\fin}(\lcro)$,
\[
\WSO(\lcro) \models \phi(\bar{A}) \Longleftrightarrow \LCI(\lcro) \models \psi(\bar{A}).
\]
Moreover $\psi(\bar{Y})$ can always be chosen to be an existential $\langLI$-formula.
\end{corollary}

\newsubsection{Interpreting \texorpdfstring{$\LCI(\lcro)$ in $\WSO(\lcro)$}{L(I) in W(I)}}

To make things more easily digestible, we will use $(\pow_{\fci}(\lcro),\subseteq)$ as an intermediary between $\LCI(\lcro)$ and $\WSO(\lcro)$.

It is straightforward, using modified versions of arguments from \cite{TresslTop}, to show that $(\pow_{\fci}(\lcro),\subseteq)$ and $\LCI(\lcro)$ have the same definable subsets. 
Using this result, together with the interpretation of $(\pow_{\fci}(\lcro),\subseteq)$ in $\WSO(\lcro)$ which we are about to give, we will indicate a particular interpretation of $\LCI(\lcro)$ in $\WSO(\lcro)$.

\begin{lemma}
Let $B,C \in \pow_{\fin}(\lcro)$. 
The following are equivalent,
\begin{enumerate}
\item there is $A \in \pow_{\fci}(\lcro) \wo \{\bot\}$ such that $A_l = B$ and $A_r = C$,
\item $B \neq \bot$, $\min(B \cup C) \subseteq B$, and one of the following holds,
\begin{enumerate}
\item $\max(B \cup C) \subseteq C \text{ and } \ips(B \cup C,C \wo B) = B \wo C$,
\item $\max(B \cup C) \subseteq B \wo C \text{ and } \ips(B \cup C,C \wo B) \cup \max(B \cup C) = B \wo C$.
\end{enumerate}
\end{enumerate}
\end{lemma}
\begin{proof}
Suppose $(1)$ holds.
Then we can rewrite $(2)$ as follows,
\begin{enumerate}[(2)]
\item $A_l \neq \bot$, $\min(A_l \cup A_r) \subseteq A_l$, and one of the following holds,
\begin{enumerate}
\item $\max(A_l \cup A_r) \subseteq A_r \text{ and } \ips(A_l \cup A_r,A_r \wo A_l) = A_l \wo A_r$,
\item $\max(A_l \cup A_r) \subseteq A_l \wo A_r \text{ and } \ips(A_l \cup A_r,A_r \wo A_l) \cup \max(A_l \cup A_r) = A_l \wo A_r$.
\end{enumerate}
\end{enumerate}
Intuitively then, (2) first says that $A$ has at least one left endpoint, and that the smallest endpoint of $A$ is a left endpoint. Then both (a) and (b) simply say that the proper left endpoints and proper right endpoints appear in pairs, with proper right endpoints immediately preceded by proper left endpoints.
An exception is needed simply for the case where $A$ is unbounded, in which case the largest endpoint is a proper left endpoint which is not the predecessor of a proper right endpoint (this is dealt with by (b)).

Conversely, suppose that $(2)$ holds. 
It is straightforward to construct $A \in \pow_{\fci}(\lcro)$ such that $A_l = B$ and $A_r = C$. 
\end{proof}

This lemma gives us the universe for our interpretation of $(\pow_{\fci}(\lcro),\subseteq)$ in $\WSO(\lcro)$.
We will identify an element $A \in \pow_{\fci}(\lcro)$ with the pair $(A_l,A_r) \in \pow_{\fin}(\lcro)$.
The lemma tells us precisely that the image of the map $\pow_{\fci}(\lcro) \rightarrow \pow_{\fin}(\lcro)^2$ given by $A \mapsto (A_l,A_r)$ is definable in $\WSO(\lcro)$. 
As moreover this map is injective, our interpretation can make use of the equality in $\WSO(\lcro)$ to interpret equality from $\LCI(\lcro)$, in particular we do not need to take a quotient of the image by a definable equivalence relation.

It remains to show that the relation $\subseteq$ on $\pow_{\fci}(\lcro)$ is interpretable in $\WSO(\lcro)$. 

\begin{lemma}
There is an $\langWI$-formula $\phi_{\in}(X_l,X_r,Z)$ such that for any $i \in \lcro$ and $A \in \pow_{\fci}(\lcro)$, $i \in A$ if and only if,
\[
\WSO(\lcro) \models \phi_{\in}(A_l,A_r,\{i\}).
\]
\end{lemma}
\begin{proof}
We split into two cases, according to whether $A \in \pow_{\fci}(\lcro)$ is bounded or unbounded (see \cref{Lemma-boundeddef}).
Let $\phi_{\bdd}(X_l,X_r,Z)$ be the $\langWI$-formula,
\[
\ips(\bdr{X} \cup Z,Z) \subseteq X_l \wo X_r \text{ and } \ips(\bdr{X} \cup Z,X_r \wo X_l) \cap Z \neq \bot,
\]
this `says' that the predecessor of $Z$ is a proper left endpoint and that $Z$ is the predecessor of a right endpoint,
Then let $\phi_{\neg \bdd}(X_l,X_r,Z)$ be the $\langWI$-formula,
\[
(\ips(\bdr{X} \cup Z,Z) \subseteq X_l \wo X_r \text{ and } \ips(\bdr{X} \cup Z,X_r \wo X_l) \cap Z \neq \bot) \text{ or } Z = \max(\bdr{X} \cup Z),
\]
this `says' that the condition in $\phi_{\bdd}$ holds or that $Z$ is greater than or equal to all left and right endpoints. 
Then for each $i \in \lcro$ and $A \in \pow_{\fci}(\lcro)$ the following are equivalent,
\begin{enumerate}
\item $i \in A$,
\item $\WSO(\lcro) \models \bdr{A} \neq \bot$ (so that $A \neq \bot$) and one of the following holds,
\begin{enumerate}
\item $\WSO(\lcro) \models \{i\} \subseteq \bdr{A}$ or,
\item $\WSO(\lcro) \models \bdd(A)$ and $\phi_{\bdd}(A_l,A_r,\{i\})$, or,
\item $\WSO(\lcro) \models \neg \bdd(A)$ and $\phi_{\neg \bdd}(A_l,A_r,\{i\})$.
\end{enumerate}
\end{enumerate}
In other words $i \in A$ if and only if $A$ is nonempty and either $i$ is an endpoint of $A$, or $i$ sits between a left and right endpoint of $A$, or $A$ is unbounded and $i$ sits in the unbounded part of $A$.
The latter we have shown is definable in $\WSO(\lcro)$ by giving $\langWI$-formulas, so $\phi_{\in}$ exists.
\end{proof}

\begin{proposition}
There is an $\langWI$-formula $\phi_{\subseteq}(X_l,X_r,Y_l,Y_r)$ such that for any $A,B \in \pow_{\fci}(\lcro)$, $A \subseteq B$ if and only if,
\[
\WSO(\lcro) \models \phi_{\subseteq}(A_l,A_r,B_l,B_r).
\]
\end{proposition}
\begin{proof}
Let $\At(Z)$ be the formula $Z \neq \bot \wedge Z = \min(Z)$.
In $\WSO(\lcro)$ this defines the collection of singletons. 
Therefore we can take for $\phi_{\subseteq}$ the $\langWI$-formula,
\[
\forall Z (\At(Z) \rightarrow (\phi_{\in}(X_l,X_r,Z) \rightarrow \phi_{\in}(Y_l,Y_r,Z))).
\]
This says that every singleton contained in $X$ is contained in $Y$, as required.
\end{proof}

\begin{theorem}
There is an interpretation of $\LCI(\lcro)$ in $\WSO(\lcro)$, for which the co-ordinate map $\pow_{\fci}(\lcro) \rightarrow \pow_{\fin}(\lcro)^2$ is given by $A \mapsto (A_l,A_r)$.
\end{theorem}

\begin{corollary}\label{WinterpretsL}
For each $\langLI$-formula $\phi(\bar{X})$ there is an $\langWI$-formula $\psi(\bar{Y})$ such that for each $\bar{A} \in \pow_{\fci}(\lcro)$,
\[
\LCI(\lcro) \models \phi(\bar{A}) \Longleftrightarrow \WSO(\lcro) \models \psi(\bar{A_l},\bar{A_r}).
\]
\end{corollary}

\newsubsection{Transfer of model-completeness from \texorpdfstring{$\WSO(\lcro)$ to $\LCI(\lcro)$}{W(I) to L(I)}}\label{Subsection-transfer}

\begin{theorem}\label{Thm-LImc}
The $\langLI$-structure $\LCI(\lcro)$ is model-complete.
\end{theorem}
\begin{proof}
Let $\phi(\bar{X})$ be an $\langLI$-formula.
We will show that over $\LCI(\lcro)$ the formula $\phi$ is equivalent to an existential formula $\phi^*$.

Using our interpretation of $\LCI(\lcro)$ in $\WSO(\lcro)$ (in particular \cref{WinterpretsL}), for each $\langLI$-formula $\phi(\bar{X})$ there is an $\langWI$-formula $\psi(\bar{Y})$ such that for each $\bar{A} \in \pow_{\fci}(\lcro)$,
\[
\LCI(\lcro) \models \phi(\bar{A}) \Longleftrightarrow \WSO(\lcro) \models \psi(\bar{A_l},\bar{A_r}). \tag{$\star$}\label{star}
\]
Without loss of generality we can take $\psi(\bar{Y})$ to be a positive existential $\langWI$-formula. 
This is because the positive model-completeness of $\WSO(\lcro)$ (\cref{Corollary-WIPositiveMC}), tells us precisely that every $\langWI$-formula is equivalent to some positive existential $\langWI$-formula over $\WSO(\lcro)$.

Using our interpretation of $\WSO(\lcro)$ in $\LCI(\lcro)$ (in particular \cref{LinterpretsW}), for each $\langWI$-formula $\psi(\bar{Y})$ there is an $\langLI$-formula $\theta(\bar{Y})$ such that for each $\bar{B} \in \pow_{\fin}(\lcro)$,
\[
\WSO(\lcro) \models \psi(\bar{B}) \Longleftrightarrow \LCI(\lcro) \models \theta(\bar{B}),
\]
and so in particular taking $\bar{B} = (\bar{A_l},\bar{A_r})$ we get,
\[
\WSO(\lcro) \models \psi(\bar{A_l},\bar{A_r}) \Longleftrightarrow \LCI(\lcro) \models \theta(\bar{A_l},\bar{A_r}). \tag{$\dagger$}\label{dagger}
\]
Moreover, as $\psi(\bar{Y})$ is a positive existential $\langWI$-formula, \cref{LinterpretsW} tells us that $\theta(\bar{Y})$ can additionally be chosen to be an existential $\langLI$-formula.
Note that without the positive model-completeness of $\WSO(\lcro)$, we could not have taken $\psi(\bar{Y})$ positive existential, and therefore could not have taken $\theta$ to be existential.

Combining $\eqref{star}$ and $\eqref{dagger}$ we get that for each $\bar{A} \in \pow_{\fci}(\lcro)$,
\[
\LCI(\lcro) \models \phi(\bar{A}) \Longleftrightarrow \LCI(\lcro) \models \theta(\bar{A_l},\bar{A_r}).
\]
Now, $l$ and $r$ are part of the signature $\langLI$, so we take $\phi^*(\bar{X})$ to be the formula $\theta(\bar{X_l},\bar{X_r})$. 
Our choice of $\phi^*$ is clearly existential, as $\theta$ is existential. 
Putting everything together, we get that,
\[
\LCI(\lcro) \models \forall \bar{X} (\phi(\bar{X}) \leftrightarrow \phi^*(\bar{X})),
\]
and hence $\LCI(\lcro)$ is model-complete. 
\end{proof}

\clearpage
\printbibliography[heading=bibintoc]

\end{document}